\documentclass{amsart}

\usepackage{amssymb}
\newcommand{\Q}{\mathbb{Q}}

\newcommand{\N}{\mathbb{N}}

\newtheorem{theorem}{Theorem}[section]
\newtheorem{lemma}[theorem]{Lemma}

\theoremstyle{definition}

\theoremstyle{theorem}
\newtheorem{corollary}[theorem]{Corollary}

\theoremstyle{theorem}
\newtheorem{proposition}[theorem]{Proposition}

\theoremstyle{theorem}

\theoremstyle{theorem}

\theoremstyle{definition}

\theoremstyle{theorem}

\numberwithin{equation}{section}

\begin{document}
\title{Asymptotic density and the Ershov hierarchy}

\author{Rod Downey}
\address{School of Mathematics, Statistics, and Operations Research\\
Victoria University of Wellington\\
New Zealand}
\email{Rod.Downey@msor.vuw.ac.nz}

\author{Carl Jockusch}
\address{Department of Mathematics\\
University of Illinois at Urbana-Champaign\\
1409 W. Green St.\\
Urbana, Illinois 61801 USA}
\email{jockusch@math.uiuc.edu}

\author{Timothy H. McNicholl}
\address{Department of Mathematics\\
Iowa State University\\
Ames, Iowa 5011}
\email{mcnichol@iastate.edu}

\author{Paul Schupp}
\address{Department of Mathematics\\
University of Illinois at Urbana-Champaign\\
1409 W. Green St.\\
Urbana, Illinois 61801 USA}
\email{schupp@math.uiuc.edu}

\begin{abstract}
  We classify the asymptotic densities of the $\Delta^0_2$ sets
  according to their level in the Ershov hierarchy.  In particular, it
  is shown that for $n \geq 2$, a real $r \in [0,1]$ is the density of
  an $n$-c.e.\ set if and only if it is a difference of left-$\Pi_2^0$
  reals.  Further, we show that the densities of the $\omega$-c.e.\ sets
  coincide with the densities of the $\Delta^0_2$ sets, and there are
  $\omega$-c.e.\ sets whose density is not the density of an  $n$-c.e. set
  for any $n \in \omega$.
\end{abstract}


\maketitle

\section{Introduction}\label{sec:INTRODUCTION}

In computability theory, the complexity of sets $A \subseteq \omega$
is often measured using Turing reducibility and  the arithmetic
hierarchy.  In number theory, the size of a set $A \subseteq \omega$
is often measured using its asymptotic density $\rho(A) \in [0,1]$,
if this density exists.  It is natural to inquire about
relationships  between these measurements.  In \cite{Downey.Jockusch.Schupp.ta} 
it is shown that there is a very tight connection between the
position of a set $A$ in the arithmetic hierarchy and the complexity
of its density $\rho(A)$ as a real number, provided that $A$ has a
density. (These results are summarized in Theorem \ref{thm:DJS} below.)
Here we measure the complexity of a real $x_0$ in terms of
the complexity of its left Dedekind cut; that is, the set of all
rational numbers smaller than $x_0$.  In the current paper we study
the corresponding relationship  when we classify $A$ according to the
Ershov hierarchy, that is,  the number of changes in a computable
approximation to $A$.

We identify sets with their characteristic functions. According to the Shoenfield Limit Lemma, 
the $\Delta^0_2$ sets $A$ are exactly
those for which there is a computable function $g$ such that, for all
$x$, $A(x) = \lim_s g(x,s)$.  Roughly speaking, the Ershov hierarchy classifies $\Delta^0_2$
sets by  the number of $s$ with $g(x,s) \neq
g(x,s+1)$. In particular, if $f$ is a function and $A \subseteq \omega$, then
$A$ is called $f$-\emph{c.e.} if there is a computable function $g$ such
that, for all $x$, $A(x) = \lim_s g(x,s)$, $g(x,0) = 0$, and
 $|\{s:g(x,s) \neq g(x,s+1)\}| \leq f(x)$.  

Our goal here is to determine the relationship between the growth rate
of $f$ and the complexity of the asymptotic density of $A$ as a real
number, if it exists.  We show that every real number which is the
density of a $\Delta^0_2$ set is  the density of an id-$c.e.$
set, where id is the identity function.  In fact, we show that the
identity function could be replaced here by any computable,
non-decreasing, unbounded function $f$.  Thus, for any such $f$  the
densities of the $f$-c.e.\ sets coincide with the densities of the
$\Delta^0_2$ sets.  Since we consider only $f$ which are computable
and nondecreasing, it remains only to consider the densities of the
$f$-c.e.\ sets in the special case where $f$ is constant.  A set $A$ is called
$n$-\emph{c.e.}\ if $A$ is $c_n$-c.e, where $c_n$ is the constant
function with value $n$ on all arguments.   Thus, for example, the $1$-c.e.\ sets are precisely the c.e.\ sets
and the $2$-c.e.\ sets are precisely the \emph{d.c.e.}\ sets; i.e.\ those sets
that are differences of two c.e.\ sets.

It is shown in Theorem 5.13 of \cite{Downey.Jockusch.Schupp.ta} that
the densities of the c.e.\ sets are precisely the left-$\Pi^0_2$ reals
in the interval $[0,1]$.  Thus  one might  expect that the
densities of the d.c.e.\ sets are precisely the differences of
left-$\Pi^0_2$ reals in $[0,1]$.  We prove that this is the case,
but  care is necessary because $A \setminus B$ can have a
density even though  $A$ and  $B$ do not have densities.  The essential
observation here is that if $B \subseteq A$ and $A \setminus B$ has a
density, then this density is $\overline{\rho}(A) -
\overline{\rho}(B)$  (where $\overline{\rho}(X)$ is the upper density
of the set $X$).  Note that a difference of left-$\Pi_n^0$ reals is also a difference of left-$\Sigma_n^0$ reals.  A difference of left-$\Sigma_1^0$ reals is also known as a \emph{d.c.e.} real.  Relativizing the proof of Corollary 4.6 of
\cite{Ambos-Spies.Weihrauch.Zheng.2000} shows that there is a real which is
a difference of left-$\Pi^0_2$ reals but which is neither left-$\Pi^0_2$ nor
left-$\Sigma^0_2$.  Combining this with our results and Theorem 5.13
of \cite{Downey.Jockusch.Schupp.ta} shows that there is a real which
is the density of a d.c.e.\ set but not the density of any c.e.\ or co-c.e.\ set.

We next  consider the densities of $n$-c.e.\ sets for arbitrary  $n \geq 2$.
It is well known that every $n$-c.e.\ set is a finite disjoint union of
d.c.e.\ sets.  Also the reals which are differences of left-$\Pi^0_2$
reals are easily seen to be closed under addition.  Indeed,  these
reals form a field, as may be seen by relativizing Theorem 3.7 of
\cite{Ambos-Spies.Weihrauch.Zheng.2000}.  Thus one might expect that
if a real $r$ is the density of an $n$-c.e.\ set, then $r$ is a difference
of left-$\Pi^0_2$ reals.   We prove this, but  care
is again necessary because a disjoint union of sets can have a density when
the sets themselves fail to have densities.    It follows that, for all
$n \geq 2$, the densities of the $n$-c.e.\ sets coincide with the densities
of the d.c.e.\ sets.

Say that a set $A$ is $\omega$-\emph{c.e.}\ if $A$ is $f$-c.e.\ for some computable function
$f$.   This hierarchy has been
extended to levels indexed by notations for arbitrary computable
ordinals (see \cite{Epstein.Haas.Kramer.1981}), but there are some
subtleties because for levels $\alpha \geq \omega^2$ the sets occurring
at level $\alpha$ depend on the choice of a notation for $\alpha$.
We show that if a $\Delta^0_2$ set has a density $r$ then $r$ is also the
density of an  $\omega$-\emph{c.e.} set.     Thus, if  $\alpha$
 is a notation for a computable
ordinal greater than or equal to $\omega$, the densities of the $\alpha$-c.e.\ sets coincide with the densities
of the $\omega$-c.e.\ sets  and these in turn coincide
with the densities of the $\Delta^0_2$ sets.

We summarize some  background and prior results needed in Section
\ref{sec:BACKGROUND}.  In Section \ref{sec:dce} we characterize the
densities of d.c.e.\ sets, and in Section \ref{sec:nce} we characterize the
densities of $n$-c.e.\ sets.   In Section \ref{sec:omegace} we show that
the densities of $\Delta^0_2$ sets coincide with the densities of
the $f$-c.e.\ sets for any computable, nondecreasing, unbounded function $f$,
    Finally in Section \ref{sec:UPPER.LOWER} we show that with respect to upper and
lower densities, the Ershov hierarchy collapses even further.

\section{Background}\label{sec:BACKGROUND}

We begin with the basic definitions related to asymptotic density.
Let $X$ be a set of natural numbers.  When $n \in \N$, let 
\[
X \upharpoonright n = \{j\ :\ j \in X\ \wedge\ j < n\}.
\]
For $n > 0$, define
\[
\rho_n(X) = \frac{|X \upharpoonright n|}{n}
\]
 The \emph{upper density of $X$}
is defined to be
\[
\overline{\rho}(X) = \limsup_n \rho_n(X) 
\]
The \emph{lower density of $X$} is defined to be 
\[
\underline{\rho}(X) = \liminf_n  \rho_n(X)
\]
If the upper and lower density of $X$ coincide, then this common value
$\rho(X) = \lim_n  \rho_n(X)$ is called the \emph{asymptotic density of $X$}.
If $\mathcal{C}$ is a complexity class (such as $\Pi_2^0$,
$\Delta_2^0$, \emph{etc.}), then a real number $r$ is \emph{left
  (right)}-$\mathcal{C}$ if and only if its left (right) Dedekind cut
belongs to $\mathcal{C}$.  So, for example, a real $r$ is left-$\Pi_2^0$ if and only if the set
\[
\{q \in \Q\ :\ q < r\}
\]
is $\Pi_2^0$.
            
 Theorem 2.21 of \cite{Jockusch.Schupp.2012} shows that the
densities of the computable sets are exactly the $\Delta^0_2$ reals in
the interval $[0,1]$.  It is shown in Theorem 5.13 of
\cite{Downey.Jockusch.Schupp.ta} that the densities of the c.e.\ sets
are exactly the left-$\Pi^0_2$ reals in $[0,1]$.  By relativizing and
dualizing these results, one easily obtains the following theorem.

\begin{theorem}\label{thm:DJS} (Downey, Jockusch, Schupp)
Let $r$ be a real number in the interval $[0,1]$.
\begin{enumerate}
         \item $r$ is the density of a $\Delta_n^0$ set if and only if $r$ is
  $\Delta_{n+1}^0$.
        \item $r$ is the density of a $\Sigma_n^0$ set if and only if $r$ is
  left-$\Pi_{n+1}^0$.
	\item $r$ is the density of a $\Pi_n^0$ set if and only if $r$ is left-$\Sigma_{n+1}^0$.
\end{enumerate}
\end{theorem}

Soare  \cite{Soare.1969}  gives many examples of real numbers which
are left-$\Sigma^0_1$ but which are not computable and hence not
left-$\Pi^0_1$. Another example of such a real is given in
\cite{Downey.Hirschfeldt.2010}, Corollary 5.1.9.  It follows by
relativization that for each $n \geq 1$ there is a real which is
left-$\Sigma^0_n$ but not left-$\Pi^0_n$.  Since a real $r$ is
left-$\Sigma^0_n$ if and only if $-r$ is left-$\Pi^0_n$, it follows
that for $n \geq 1$ the left-$\Sigma^0_n$ reals are not closed under
subtraction.
To obtain closure under subtraction, we instead consider reals of the
form $r - s$ where the reals $r$ and $s$ are left-$\Pi^0_n$ reals.
Let $\mathcal{D}_n$ be the set of such reals.  Study of the class
$\mathcal{D}_1$ was initiated in
\cite{Ambos-Spies.Weihrauch.Zheng.2000}, where elements of
$\mathcal{D}_1$ are called \emph{weakly computable reals}. That paper
 shows that $\mathcal{D}_1$ is actually a field.
It was further shown by Ng and  independently by Raichev that
$\mathcal{D}_1$ is a real-closed field.  (Proofs of these statements  are  also 
given in Chapter 5  of \cite{Downey.Hirschfeldt.2010} which is
a comprehensive source of information on the subject.)  The cited 
results extend by relativization to  $\mathcal{D}_n$ for all $n \geq 1$. 

By the remarks above, there are reals in $\mathcal{D}_n$ which are
neither left-$\Sigma^0_n$ nor left-$\Pi^0_n$.  Also, Ambos-Spies,
Weihrauch, and Zheng (\cite{Ambos-Spies.Weihrauch.Zheng.2000},
Corollary 4.10) showed that there is a $\Delta^0_2$ real which is not
in $\mathcal{D}_1$.  It again follows by relativization that for
each $n \geq 1$ there is a $\Delta^0_{n+1}$-real which is not in
$\mathcal{D}_n$.

Let $\mathcal{D}_n^A$ be the class of reals which are differences of
left-$\Pi^{0,A}_n$ reals, so $\mathcal{D}_n^A$ is simply the
relativization of $\mathcal{D}_n$ to $A$. Such 
relativized classes play a useful role in algorithmic randomness.
Call a set $A$ \emph{low for} $\mathcal{D}_n$ if $\mathcal{D}_n^A =
\mathcal{D}_n$.  It was shown by J. Miller (see Theorem 15.9.2 of
\cite{Downey.Hirschfeldt.2010}) that the $K$-trivial sets in the
sense of algorithmic randomness  are precisely the sets $A$ which are
low for $\mathcal{D}_1$.

\section{Densities of d.c.e.\ sets} \label{sec:dce}

It is shown in  Theorem 5.13 of \cite{Downey.Jockusch.Schupp.ta} that the
densities of the c.e.\ sets are the left-$\Pi^0_2$ reals in $[0,1]$.  Hence if
$A$, $B$ are c.e.\  sets having densities and $B \subseteq A$, then
$\rho(A \setminus B) = \rho(A) - \rho(B)$ is a difference of
left-$\Pi^0_2$ reals.  This suggests that if $r$ is the density of a
d.c.e.\ set, then  $r$ should also be  a difference of
left-$\Pi^0_2$ reals, that is, $r \in \mathcal{D}_2$.  However, 
 $A \setminus B$ can have a density even when $B \subseteq A$, and the
sets $A$ and $B$ do not have densities.  Nonetheless,
we will prove in this section that if a d.c.e.\ set has density
$r$, then $r \in \mathcal{D}_2$.  Conversely, we show that every real
in $\mathcal{D}_2 \cap [0,1]$ is the density of a d.c.e.\ set, thus
characterizing the densities of the d.c.e.\ sets as the reals in
$\mathcal{D}_2 \cap [0,1]$.  This  implies that there is a real
which is the density of a d.c.e.\ set but not of any c.e.\ or
co-c.e.\ set.

 The following proposition shows that we can use upper densities
to avoid the above mentioned difficulty  of nonexistent densities. 
 
\begin{proposition}\label{prop:limsup}
If $M \geq a_n \geq b_n \geq L$ for all $n$, and if $\lim_{n \rightarrow \infty} (a_n - b_n)$ exists, then 
$\lim_{n \rightarrow \infty} (a_n - b_n) = \limsup_n a_n - \limsup_n b_n$.
\end{proposition}

\begin{proof}

  Note that the result is clear if $a_n - b_n$ is constant, since then
  $\{a_n\}$ and $\{b_n\}$ are near their respective lim sups
  simultaneously, and so respective lim sups must differ by the same
  constant.  We show below that essentially this same argument works
  when we assume only that $a_n - b_n$ has a limit.

Let $a = \limsup_n a_n$ and $b = \limsup_n b_n$, where these are real numbers  because
the given sequences are bounded.   Let $d = \lim_n(a_n - b_n)$, which
exists by hypothesis.   We must show that $d = a - b$, which we prove in the
form $b = a - d$, i.e.  $\limsup_n b_n = a - d$.

Let $\epsilon > 0$ be given.   Since $\limsup_n a_n = a$, we have $a_n \leq a + \epsilon/2$
for all sufficiently large $n$.   Since $\lim_n (b_n - a_n) = -d$, we also have
$b_n - a_n \leq  -d + \epsilon/2$ for all sufficiently large $n$.   Adding these inequalities,
we have $b_n \leq a - d + \epsilon$ for all sufficiently large $n$.   Since $\epsilon$ was arbitrary,
we conclude that $b = \limsup_n b_n \leq a - d$.

To obtain the reverse inequality, again let $\epsilon > 0$ be given.
Since $\limsup_n a_n = a$, there are infinitely many $n$ such that
$a_n \geq a - \epsilon/2$.  Let $S$ be the set of such $n$.  Since
$\lim_n (b_n - a_n) = -d$, we have $b_n - a_n \geq -d - \epsilon/2$ for
all sufficiently large $n$.  Adding these inequalities, we have that
$b_n \geq a - d - \epsilon$ for all sufficiently large $n \in S$, and
hence for infinitely many $n$.  Since $\epsilon$ was arbitrary, we
conclude that $b = \limsup_n b_n \geq a - d$, and hence, by the previous 
paragraph, $b = a - d$.
\end{proof}

\begin{corollary}
  If $Y$ is a subset of $X$, and if $X - Y$ has a density, then its
  density is the upper density of $X$ minus the upper density of $Y$.
\end{corollary}

\begin{corollary} \label{cor:dceD2} If $C$ is a d.c.e.\ set which has a density, then
  $\rho(C) \in \mathcal{D}_2$.
\end{corollary}

\begin{proof}
Let $C = A \setminus B$, where $A, B$ are c.e.\ and $B \subseteq A$.  Then
$\rho(C) = \overline{\rho}(A) - \overline{\rho}(B)$ by the previous corollary
and the reals $\overline{\rho}(A)$ and $\overline{\rho}(B)$ are each left-$\Pi^0_2$
by \cite{Downey.Jockusch.Schupp.ta}, Theorem 5.6.
\end{proof}   

The next  theorem will  allow us to prove the  converse: Every real
in $\mathcal{D}_2 \cap [0,1]$ is the density of a d.c.e.\ set.
In order to prove the theorem we need the following lemma which asserts a
well-known fact about conditional densities.

\begin{lemma}
Let $h$ be a strictly increasing function and let $X \subseteq \omega$.
Then $\rho(h(X)) = \rho(\mbox{range}(h))\rho(X)$ provided that both the
range of $h$ and $X$ have densities.
\end{lemma}

\begin{proof}

Let $R$ be the range of $h$, and for each $u$, let $g(u)$ be the least
$k$ such that $h(k) \geq u$.   Note that, for all $u$,
$$|h(X) \upharpoonright u| = |X \upharpoonright g(u)|  \quad \& \quad
|R \upharpoonright u| = g(u)$$
via the bijections induced by $h$.   It follows that
$$\rho_{g(u)}(X) \cdot \rho_u(R) = \frac{|X \upharpoonright g(u)|}{g(u)} \ 
\frac{|R \upharpoonright u|}{u} = \frac{|h(X) \upharpoonright
  u|}{g(u)} \ \frac{g(u)}{u} = \frac{|h(X) \upharpoonright u|}{u} =
\rho_{u}(h(X))$$ for all $u$.  Hence, $\rho_{u}(h(X)) = \rho_{g(u)}(X)
\cdot \rho_u(R)$.  As $u$ tends to infinity, $g(u)$ also tends to
infinity, and the lemma follows.
\end{proof}

\begin{theorem}\label{thm:difference}
  If $a,b$ are left-$\Pi_2^0$ reals such that $0 \leq b \leq a \leq 1$, 
  then there is a c.e.\ set $A$ with density $a$ and a c.e.\ set $B \subseteq A$ with density $b$.
\end{theorem}

\begin{proof}
  It is shown in Theorem 5.13 of \cite{Downey.Jockusch.Schupp.ta} that
  every left-$\Pi^0_2$ real in the interval $[0,1]$ is the density of
  a c.e.\ set, which is the case $a = b$ of the current result.  Thus,
  we may assume that $b < a$.  Let $q$ be a rational number such
  that $b < q < a$, and let $C$ be a computable set of density $q$,
  which exists by Theorem 2.21 of \cite{Jockusch.Schupp.2012}.  We will
  obtain $A$ by expanding $C$ and obtain $B$ by shrinking $C$.  In more
  detail, we obtain $A$ as $C \cup A_0$, where $A_0 \subset
  \overline{C}$ is a c.e.\ set of density $a - q$.

Let $h$ be a computable, strictly increasing  function with range $\overline{C}$.  
Then let $A_0 = h(A_1)$, where $A_1$ is a c.e.\ set of density $(a- q)/(1 - q)$.  
Such a set exists by Theorem 5.13 of \cite{Downey.Jockusch.Schupp.ta} 
because $(a- q)/(1 - q)$ is a left-$\Pi^0_2$ real in $[0,1]$.  
Hence, 
$$\rho(A_0) = \rho(h(A_1)) = (1-q)\frac{a-q}{1-q} = a-q$$
 by the lemma, and thus
$$\rho(A) = \rho(C \sqcup A_0) = \rho(C) + \rho(A_0) = q + (a - q) = a$$
as desired.  The c.e.\ set $B \subseteq C$ of density $b$ is obtained analogously, but
working within $C$ instead of $\overline{C}$. Namely $B = h(B_1)$,
where $h$ is now a strictly increasing computable function with range
$C$ and $B_1$ is a c.e.\ set of density $b/q$.  
Since  $B \subseteq C \subseteq A$, the proof is complete.

\end{proof}

\begin{corollary}   The densities of the d.c.e.\ sets coincide with the reals
in $\mathcal{D}_2 \cap [0,1]$.
\end{corollary}
\begin{proof}
 The density of a d.c.e.\ set is  in $\mathcal{D}_2 \cap [0,1]$  by Corollary \ref{cor:dceD2}. 
 For the other direction, consider a real $r \in [0,1]$ which is a
  difference of left-$\Pi^0_2$ reals.  Write $r$ as $a - b$, where $a,b$ are
  left-$\Pi^0_2$ reals and $1 \geq a \geq b \geq 0$.  By Theorem
  \ref{thm:difference}, there are c.e.\ sets $A, B$ such that $B
  \subseteq A$,  $\rho(A) = a$, and $\rho(B) = b$.
   Then $A \setminus B$ is d.c.e. and  $\rho(A  \setminus B) = a - b$.
\end{proof}

\begin{corollary}\label{cor:NO.COLLAPSE}
  There is a $2$-c.e.\ set which  has a density but whose density is not
  the density of any c.e.\ set or co-c.e.\ set.
\end{corollary}

\begin{proof}
  By Corollary 4.6 of \cite{{Ambos-Spies.Weihrauch.Zheng.2000}},
  relativized to $0'$, there is a real $r$ which is a difference of
  left-$\Pi^0_2$ reals but is not left-$\Pi^0_2$ or left-$\Sigma^0_2$.
We may assume that $r \in [0,1]$, so $r$ is the density of a $2$-c.e.
set.  The real $r$ is not the density of a c.e.\ or co-c.e.\ set, since
the densities of c.e.\ sets are left-$\Pi^0_2$ and the densities of
co-c.e.\ sets are left-$\Sigma^0_2$.
\end{proof}

\section{The densities of $n$-c.e.\ sets} \label{sec:nce}

It is well known that if $D$ is an $n$-c.e.\ set then, for some $k$,
 $D = D_1 \cup D_2 \cup \dots \cup D_k$  where $D_1, D_2, \dots, D_k$ are
pairwise disjoint d.c.e.\ sets.  If each $D_i$ has a density, then
$\rho(D) = \sum_{i \leq k} \rho(D_i)$, where $\rho(D_i) \in
\mathcal{D}_2$ by Corollary \ref{cor:dceD2}.  Since $\mathcal{D}_2$ is
closed under addition, it follows that $\rho(D) \in \mathcal{D}_2$.
However, we again have the situation that 
a disjoint union of sets can have a density when the sets themselves
do not.   This time, an algebraic trick will come  to our rescue. 

The following proposition is a well-known fact about $n$-c.e. sets.

\begin{proposition}\label{prop:union}
Suppose $A$ is an $n$-c.e.\ set where $n$ is a positive integer.  
\begin{enumerate}
	\item If $n = 2k$ where $k \in \N$, then $A$ can be written in the form 
	\[
	(A_1 - A_2) \cup \ldots \cup (A_{2k - 1} - A_{2k})
	\]
	where $A_1, \ldots, A_{2k}$ are c.e.\ and $A_1 \supseteq \ldots \supseteq A_{2k}$.
	
	\item If $n = 2k + 1$ where $k \in \N$, then $A$ can be written in the form
	\[
	(A_1 - A_2) \cup \ldots \cup (A_{2k- 1} - A_{2k}) \cup A_{2k+ 1}
	\]
	where $A_1, \ldots, A_{2k+1}$ are c.e.\ and $A_1 \supseteq \ldots \supseteq A_{2k+1}$.
\end{enumerate}
\end{proposition}

\begin{theorem} \label{thm:nce}
  If $n \geq 1$, and if $A$ is an $n$-c.e.\ set that has a density, then
  the density of $A$ is a difference of left-$\Pi^0_2$ reals.
\end{theorem}

\begin{proof}
  Without loss of generality, suppose $n = 2k$ where $k \in \N$.  By
  Proposition \ref{prop:union}, there are c.e.\ sets $A_1, \ldots,
  A_{2k}$ such that $A = (A_1 - A_2) \cup \ldots \cup (A_{2k - 1} -
  A_{2k})$ and $A_1 \supseteq \ldots \supseteq A_{2k}$.  Thus, $A_1 -
  A_2, \ldots, A_{2k - 1} - A_{2k}$ are pairwise disjoint.  Let
  $a_{j,s} = \rho_s(A_j)$.  It follows that
\[
\rho_s(A) = \left( \sum_{j\ \mbox{odd}} a_{j,s} \right) - \left(\sum_{j\ \mbox{even}} a_{j,s} \right).
\]
Note that $a_{2, s} \leq a_{1, s}$, $a_{4,s} \leq a_{3, s}$, $\ldots$, $a_{2k, s} \leq a_{2k - 1, s}$.
So, by Proposition \ref{prop:limsup}, $\rho(A) = a - b$ where
\begin{eqnarray*}
a & = & \limsup_s \left( \sum_{j\ \mbox{odd}} a_{j,s} \right)\mbox{, and where}\\
b & = & \limsup_s \left(\sum_{j\ \mbox{even}} a_{j,s} \right).
\end{eqnarray*}
It thus  suffices to show that if $\{q_n\}$ is a computable sequence of
rational numbers and $r = \limsup_n q_n$, then $r$ is a left-$\Pi^0_2$ real.
This is obvious if $r$ is itself rational.   Otherwise, for every rational
number $q$, $q < r$ if and only if there are infinitely many $n$ with
$q < q_n$, from which the claim follows.

\end{proof}

\begin{corollary} Let $n \geq 2$.  The densities of the $n$-c.e.\ sets
  coincide with the reals in $\mathcal{D}_2 \cap [0,1]$ and hence with
  the densities of the $2$-c.e.\ sets.
\end{corollary}

\section{Densities of $\omega$-c.e.\ sets}\label{sec:omegace}

It is shown in \cite{Jockusch.Schupp.2012} that the densities of the
computable sets are precisely the $\Delta^0_2$ reals in $[0,1]$.  By
relativization, the densities of the $\Delta^0_2$ sets are precisely
the $\Delta^0_3$ reals in $[0,1]$.  In this section, we show that the
densities of the $\omega$-c.e.\ sets coincide with the densities of
the $\Delta^0_2$ sets and in fact prove the following much stronger
result.

\begin{theorem}\label{thm:collapse.1}
  Let $f$ be a computable, nondecreasing, unbounded function.  If $A$
  is a $\Delta^0_2$ set that has a density, then the density of $A$ is
  that of an $f$-c.e.\ set.
\end{theorem}

\begin{proof}

  We must construct an $f$-c.e.\ set $B$ such that $\rho(B) = \rho(A)$.
Our definition of $B$ uses an oracle for $0'$ and also a computable
approximation $\{A_s\}$ to $A$.
  We will define an increasing modulus function $m$ for $A$, and
  arrange that, for each $x>0$, 

$$|\rho_{m(x)}(B) - \rho_x(A)| \leq  1/x$$  

  It then  follows that $\rho(B) = \rho(A)$ if $B$ has
  density. We show that  $B$ can be defined  on arguments not in the
  range of $m$ in such a way  that $B$ does in fact have a density.

  We now define $m$ by recursion.  Let $m(0) = (\mu s)[f(s) > 0]$.
  Given $m(x)$, let $m(x+1)$ be the least $s > (x+1) \cdot m(x)$ such
  that:
$$f(s) > x+1 \ \& \ (\forall t \geq s)[A_t \upharpoonright (x+1) = A_s \upharpoonright (x+1)]$$
Note that $m$ is total because $f$ is unbounded.

We now define $B(y)$ by recursion on $y$.  If $y < m(0)$, then $y
\notin B$.  Now suppose $y \in [m(x), m(x+1))$, and $B(z)$ has been
defined for all $z < y$, so $\rho_y(B)$ is defined.  Then put $y$ into
$B$ if and only if $\rho_y(B) < \rho_{x+1}(A)$.  The intuition is that
we are increasing the density of $B$ when it is less than or equal to
its ``target value'' $\rho_{x+1}(A)$ and otherwise we are decreasing
it.  Hence, as $y$ increases toward $m(x+1)$, $\rho_y(B)$ should move
in the direction of this target value, and not stray far from it once
it gets close to it.

To make this argument more precise, consider first the case where
$\rho_{m(x)}(B) \leq \rho_{x+1}(A)$.  Let $y_0$ be the least element
$y$ of $\overline{B}$ in the interval $[m(x), m(x+1))$, or $y_0 =
m(x+1)$ if there is no such $y$.  Then $\rho_y(B)$ is increasing in
$y$ for $y \in [m(x), y_0)$ and $|\rho_{y_0}(B) - \rho_{x+1}(A) |
\leq 1/(x+1)$.  Further, it is easy to see by induction on $y$ that
$|\rho_y(B) - \rho_{x+1}(A)| \leq 1/(x+1)$ for $y_0 \leq y \leq
m(x+1)$.  Thus, $|\rho_{m(x+1)}(B) - \rho_{x+1}(A)| \leq 1/(x+1)$.
In addition, for all $y \in [m(x), m(x+1))$, either $\rho_{m(x)}(B) \leq
\rho_y(B) \leq \rho_{x+1}(A)$ or $|\rho_y(B) - \rho_{x+1}(A)| \leq
1/(x+1)$, as can be seen by considering the cases $y \leq y_0$ and $y
> y_0$.  Dual considerations show that if $\rho_{m(x)}(B) \geq
\rho_{x+1}(A)$, then again $|\rho_{m(x+1)}(B) - \rho_{x+1}(A)| \leq
1/(x+1)$.   Also, for all $y \in [m(x), m(x+1))$, either $\rho_{m(x)}(B) \geq
\rho_y(B) \geq \rho_{x+1}(A)$ or $|\rho_y(B) -
\rho_{x+1}(A)| \leq 1/(x+1)$.

From the above, it follows at once that $\lim_x \rho_{m(x)}(B)
= \lim_x \rho_x(A) = \rho(A)$.    Further, if $\rho_{m(x)}(B)$
and $\rho_{m(x+1)}(B)$ are both within $\epsilon$ of $\rho(A)$, then for
all $y \in [m(x), m(x+1))$, $\rho_y(B)$ is within $\epsilon + 1/(x+1)$
of $\rho(A)$  by the above paragraph.   Hence, $\rho(B) = \lim_y \rho_y(B)
= \lim_x \rho_{m(x)}(B) = \rho(A)$.

We now show that $B$ is $f$-c.e.\  First, observe that $B$ is
$\Delta^0_2$ since $A$ and $m$ are $\Delta^0_2$, $f$ is computable,
and the sets $A_s$ are uniformly computable.  Thus $B$ has a
computable approximation $\{B_s\}$.  Further, if $A_0 = \emptyset$
and we choose $\{B_s\}$ in a natural way starting with our given
approximation $\{A_s\}$ to $A$, then $B_0 = \emptyset$ and for each
$z$ there are at most $f(z)$ values of $s$ with $B_{s+1}(z) \neq
B_s(z)$.  This implies that $B$ is $f$-c.e.\  The proof is a
straightforward argument which we merely sketch.  Call a function $h$
\emph{approximable from below} if there is a computable function $g$
such that $h(x) = \lim_s g(x,s)$ for all $x$ and $g(x,s) \leq
g(x,s+1)$ for all $x$ and $s$.  It is easy to see that the function
$m$ defined above is approximable from below.  Let $h(z)$ be the least
$x$ with $z < m(x)$.  Define ``approximable from above''
analogously. Then $h$ is approximable from above because $m$ is
approximable from below.  Further, note that $z < m(f(z))$ for all
$z$, by the definition of $m$.  It follows, by the definition of $h$,
that $h(z) \leq f(z)$ for all $z$.  Hence $h$ is approximable from
above via a function $g$ with $g(z,0) = f(z)$ for all $z$.  It follows
that for each $z$ there are at most $f(z)$ values of $s$ with
$g(z,s+1) \neq g(z,s)$.  Crucially, if we define the approximation $g$
in a natural way, and if $g(z, s+1) = g(z,s)$, then $B_{s+1}(z) = B_s(z)$.
This is because, if $z \in [m(x), m(x+1))$ (so $h(z) = x+1$) then
$B(z)$ is determined by $A \upharpoonright (x+1)$, so if our
approximation to the value of $x+1$ does not change, our approximation
to $m(x)$ does not change either, and so our approximation to $A
\upharpoonright (x+1)$ does not change either.  Since $B(z)$ is
determined by our approximations to $h(z)$ and $A \upharpoonright
h(z)$, it follows that our approximation to $B(z)$ does not change
if our approximation to $g(z)$ does not change.   Since our approximation
to $g(z)$ changes at most $f(z)$ times, our approximation to $B(z)$
changes at most $f(z)$ times, and hence $B$ is $f$-c.e.

\end{proof}

In the above proof, we assumed that $A$ had a density.   However, the
same proof establishes the following stronger result, where we make
no such assumption. 

\begin{corollary}(to proof)
For any computable, nondecreasing, unbounded function $f$ and any
$\Delta^0_2$ set $A$, there is an $f$-c.e.\ set $B$ such that
$\overline{\rho}(B) = \overline{\rho}(A)$ and $\underline{\rho}(B) =
\underline{\rho}(A)$.
\end{corollary}

\begin{corollary}\label{cor:omegasep}
  For any computable, nondecreasing, unbounded function $f$ there is
  an $f$-c.e.\ set that has a density, but its density is not the
  density of any $n$-c.e.\ set,   $n \in \omega$.
\end{corollary}

\begin{proof}
  By Corollary 4.10 of \cite{Ambos-Spies.Weihrauch.Zheng.2000}, relative
  to $0'$, there is a $\Delta^0_3$ real $r$ in the interval $[0,1]$
  which is not a difference of left-$\Pi^0_2$ reals. 
  Thus, by Theorem \ref{thm:nce}, $r$ is not the density of any
  $n$-c.e.\ set for any $n$.  On the other hand, by Theorem 2.21 of
  \cite{Jockusch.Schupp.2012}, relativized to $0'$, there is a
  $\Delta^0_2$ set $A$ with density $r$.  Then by Theorem \ref{thm:collapse.1},
there is an $f$-c.e.\ set $B$ of density $r$.
\end{proof}

\section{Upper and lower density}\label{sec:UPPER.LOWER}

Since, with respect to density, the Ershov hierarchy collapses to
levels $0$, $1$, $2$, and $\omega$, it is natural to ask if 
there is  any more separation with respect to upper and lower densities.
The following observations show that in some sense we get even more
collapse.

\begin{proposition}\label{prop:UPPER}
Let $r \in [0,1]$.  Then, the following are equivalent.
\begin{enumerate}
	\item $r$ is the upper density of a $\Sigma^0_2$ set.
	\item $r$ is left-$\Pi^0_3$.
	\item $r$ is the upper density of a $\Pi_1^0$ set.
\end{enumerate}
\end{proposition}

\begin{proof} Without loss of generality, let us assume
  $r$ is irrational.

  Suppose $r$ is the upper density of a $\Sigma_2^0$ set.  Then, by
  Theorem 5.7 of \cite{Downey.Jockusch.Schupp.ta} relativized to
  $\emptyset'$, $r$ is left-$\Pi^0_3$.

At the same time, if $r$ is left-$\Pi_3^0$, then $1 - r$ is 
left-$\Sigma_3^0$.  So, by Theorem 5.8 of
\cite{Downey.Jockusch.Schupp.ta}, $1 - r$ is the lower density of a
c.e.\ set.  It follows that $r$ is the upper density of a co-c.e.\ set.
The remaining implication is immediate.  
\end{proof}

\begin{corollary}\label{cor:UPPER}
  Let $r \in [0,1]$.  Then, for all $n \geq 2$, $r$ is the upper
  density of an $n$-c.e.\ set if and only if $r$ is the upper density
  of a co-c.e.\ set.
\end{corollary}

The above results can be dualized to show that the lower densities
of the $\Pi^0_2$ sets coincide with the left-$\Sigma^0_3$ reals
in $[0,1]$ and with the lower densities of c.e.\ sets and  we thus
have a similar collapse for lower densities.   In particular,
if $A$ is any $\Delta^0_2$ set, there is a c.e.\ set with the same
lower density as $A$ and a co-c.e.\ set with the same upper density
as $A$.

\section{Summary}

We have shown that the densities of the $2$-c.e.\ sets coincide with
the reals in $[0,1]$ which are differences of left-$\Pi^0_2$ reals,
and hence there is a real which is the density of a $2$-c.e.\ set but
not of any c.e.\ or co-c.e.\ set.  We have also proved that, for $n \geq
2$ the densities of the $n$-c.e.\ sets coincide with the densities of
the $2$-c.e.\ sets.  Finally, we have shown that if $A$ is a
$\Delta_2^0$ set that has a density, then its density is the density
of an $\omega$-c.e.\ set, and in fact the density of an $f$-c.e.\ set
for each computable, nondecreasing, unbounded function $f$.   It follows
that for each such $f$ there is a real number which is the density of
an $f$-c.e.\ set but not of any $n$-c.e.\ set,  $n \in \omega$.

%\section*{Acknowledgement}

\bibliographystyle{amsplain}
%\bibliography{/Users/tim/myfolders/research/bibliographies/analysis,/Users/tim/myfolders/research/bibliographies/computability}
\def\cprime{$'$}
\providecommand{\bysame}{\leavevmode\hbox to3em{\hrulefill}\thinspace}
\providecommand{\MR}{\relax\ifhmode\unskip\space\fi MR }
% \MRhref is called by the amsart/book/proc definition of \MR.
\providecommand{\MRhref}[2]{%
  \href{http://www.ams.org/mathscinet-getitem?mr=#1}{#2}
}
\providecommand{\href}[2]{#2}

\end{document}